\documentclass{article}%
\usepackage{graphicx}
\usepackage[intlimits]{amsmath}
\usepackage{latexsym}
\usepackage{amsfonts}
\usepackage{amssymb}%
\makeatletter
\newfont{\footsc}{cmcsc10 at 8truept}
\newfont{\footbf}{cmbx10 at 8truept}
\newfont{\footrm}{cmr10 at 10truept}
\newtheorem{theorem}{Theorem}

\newtheorem{conjecture}[theorem]{Conjecture}
\newtheorem{corollary}[theorem]{Corollary}

\newtheorem{definition}[theorem]{Definition}

\newtheorem{lemma}[theorem]{Lemma}

\newtheorem{remark}[theorem]{Remark}

\newenvironment{proof}[1][Proof]{\noindent{\textbf {#1}  }}  {\hfill$\Box$\bigskip}

\begin{document}

\title{The\ proof of a conjecture on largest Laplacian and signless Laplacian
H-eigenvalues of uniform hypergraphs\thanks{This work was supported by the
Hong Kong Research Grant Council (Grant Nos. PolyU 502111, 501212, 501913 and
15302114) and NSF of China (Grant Nos. 11231004, 11271288 and 11101263) and by
a grant of \textquotedblleft The First-class Discipline of Universities in
Shanghai\textquotedblright.\textit{ }}}
\author{Xiying Yuan\thanks{Department of Mathematics, Shanghai University, Shanghai
200444, China; \textit{E-mail address: xiyingyuan2007@hotmail.com }} \ Liqun
Qi\thanks{Department of Applied Mathematics, The Hong Kong Polytechnic
University, HungHom, Kowloon, HongKong, \textit{Email address:
liqun.qi@polyu.edu.hk}} \ Jiayu Shao\thanks{Corresponding author}
\thanks{Department of Mathematics, Tongji University, Shanghai, China,
\textit{Email address: jyshao@tongji.edu.cn}}}
\maketitle

\begin{abstract}
Let $\mathcal{A(}G\mathcal{)},\mathcal{L(}G\mathcal{)}$ and $\mathcal{Q(}%
G\mathcal{)}$ be the adjacency tensor, Laplacian tensor and signless Laplacian
tensor of uniform hypergraph $G$, respectively. Denote by $\lambda
(\mathcal{T})$ the largest H-eigenvalue of tensor $\mathcal{T}$. Let $H$ be a
uniform hypergraph, and $H^{\prime}$ be obtained from $H$ by inserting a new
vertex with degree one in each edge. We prove that $\lambda(\mathcal{Q(}%
H^{\prime}\mathcal{)})\leq\lambda(\mathcal{Q(}H\mathcal{)}).$ Denote by
$G^{k}$ the $k$th power hypergraph of an ordinary graph $G$ with maximum
degree $\Delta\geq2$. We will prove that $\{\lambda(\mathcal{Q(}%
G^{k}\mathcal{)})\}$ is a strictly decreasing sequence, which imply Conjectrue
4.1 of Hu, Qi and Shao in \cite{HuQiShao2013}. We also prove that
$\lambda(\mathcal{Q(}G^{k}\mathcal{)})$ converges to $\Delta$ when $k$ goes to
infinity. The definiton of $k$th power hypergraph $G^{k}$ has been generalized
as $G^{k,s}.$ We also prove some eigenvalues properties about $\mathcal{A(}%
G^{k,s}\mathcal{)},$ which generalize some known results. Some related results
about $\mathcal{L(}G\mathcal{)}$ are also mentioned$.$



\end{abstract}

\section{Introduction}

Let $G$ be an ordinary graph, and $A(G)$ be the adjacency matrix of $G$. We
denote the set $\{1,2,\cdot\cdot\cdot,n\}$ by $[n].$ Hypergraph is a natural
generalization of ordinary graph (see \cite{Berge}). A hypergraph
$G=(V(G),E(G))$ on $n$ vertices is a set of vertices, say $V(G)=\{1,2,\cdot
\cdot\cdot,n\}$ and a set of edges, say $E(G)=\{e_{1},e_{2},\cdot\cdot
\cdot,e_{m}\},$ where $e_{i}=\{i_{1},i_{2},\cdots,i_{l}\},i_{j}\in\lbrack n],$
$j=1,2,\cdots,l.$ If $|e_{i}|=k$ for any $i=1,2,\cdot\cdot\cdot,m,$ then $G$
is called a $k$-uniform hypergraph. In particular, the 2-uniform hypergraphs
are exactly the ordinary graphs. For a vertex $v\in V(G)$ the degree
$d_{G}(v)$ is defined as $d_{G}(v)=|\{e_{i}:v\in e_{i}\in E(G)\}|.$ Vertex
with degree one is called pendent vertex in this paper.

An order $k$ dimension $n$ tensor $\mathcal{T=}(\mathcal{T}_{i_{1}i_{2}\cdots
i_{k}})\in\mathbb{C}^{n\times n\times\cdots\times n}$ is a multidimensional
array with $n^{k}$ entries, where $i_{j}\in\lbrack n]$ \ for each
$j=1,2,\cdot\cdot\cdot,k.$

To study the properties of uniform hypergraphs by algebraic methods, adjacency
matrix, signless Laplacian matrix and Laplacian matrix of graph are
generalized to adjacency tenor, signless Laplacian tensor and Laplacian tensor
of uniform hypergraph.

\begin{definition}
\cite{HuQiXie2015} \cite{Qi2014}. Let $G=(V(G),E(G))$ be a $k$-uniform
hypergraph on $n$ vertices. The adjacency tensor of $G$ is defined as the
$k$-$th$ order $n$-dimensional tensor $\mathcal{A}(G)$ whose $(i_{1}\cdots
i_{k})$-entry is:
\[
(\mathcal{A}(G))_{i_{1}i_{2}\cdots i_{k}}=%
\begin{cases}
\frac{1}{(k-1)!} & \text{if $\{i_{1},i_{2},\cdots,i_{k}\}\in E(G),$}\\
0 & \text{otherwise}.
\end{cases}
\]
Let $\mathcal{D}(G)$ be a $k$-$th$ order $n$-dimensional diagonal tensor, with
its diagonal entry $\mathcal{D}_{ii\cdots i}$ the degree of vertex $i$, for
all $i\in\lbrack n]$. Then $\mathcal{Q(}G\mathcal{)}=\mathcal{D(}%
G\mathcal{)}+\mathcal{A(}G\mathcal{)}$ is the signless Laplacian tensor of the
uniform hypergraph $G,$ and $\mathcal{L(}G\mathcal{)}=\mathcal{D(}%
G\mathcal{)}-\mathcal{A(}G\mathcal{)}$ is the Laplacian tensor of the uniform
hypergraph $G$.
\end{definition}

The following general product of tensors, was defined in \cite{Shao} by Shao,
which is a generalization of the matrix case.

\begin{definition}
Let $\mathcal{A}\in\mathbb{C}^{n_{1}\times n_{2}\times\cdots\times n_{2}}$ and
$\mathcal{B}\in\mathbb{C}^{n_{2}\times n_{3}\times\cdots\times n_{k+1}}$ be
order $m\geq2$ and $k\geq1$ tensors, respectively. The product $\mathcal{AB}$
is the following tensor $\mathcal{C}$ of order $(m-1)(k-1)+1$ with entries:
\begin{equation}
\mathcal{C}_{i\alpha_{1}\cdots\alpha_{m-1}}=\sum_{i_{2},\cdots,i_{m}\in\lbrack
n_{2}]}\mathcal{A}_{ii_{2}\cdots i_{m}}\mathcal{B}_{i_{2}\alpha_{1}}%
\cdots\mathcal{B}_{i_{m}\alpha_{m-1}}, \label{1}%
\end{equation}
where $i\in\lbrack n],\alpha_{1},\cdots,\alpha_{m-1}\in\lbrack n_{3}%
]\times\cdots\times\lbrack n_{k+1}]$.
\end{definition}

Let $\mathcal{T}$ be an order $k$ dimension $n$ tensor, let $x=(x_{1}%
,\cdot\cdot\cdot,x_{n})^{T}\in\mathbb{C}^{n}$ be a column vector of dimension
$n$. Then by (1) $\mathcal{T}x$ is a vector in $\mathbb{C}^{n}$ whose $i$th
component is as the following%

\begin{equation}
(\mathcal{T}x)_{i}=\sum_{i_{2},\cdots,i_{k}=1}^{n}\mathcal{T}_{ii_{2}\cdots
i_{k}}x_{i_{2}}\cdots x_{i_{k}}.
\end{equation}

Let $x^{[k]}=(x_{1}^{k}, \cdots, x_{n}^{k})^{T}$. Then (see \cite{ChangPZ}
\cite{Qi2014}) a number $\lambda\in\mathbb{C}$ is called an eigenvalue of the
tensor $\mathcal{T}$ if there exists a nonzero vector $x \in\mathbb{C}^{n}$
satisfying the following eigenequations%

\begin{equation}
\mathcal{T}x=\lambda x^{[k-1]}, \label{2}%
\end{equation}
and in this case, $x$ is called an eigenvector of $\mathcal{T}$ corresponding
to eigenvalue $\lambda$.

An eigenvalue of $\mathcal{T}$ is called an H-eigenvalue, if there exists a
real eigenvector corresponding to it (\cite{Qi2014}). In this paper we will
focus on the largest H-eigenvalue of tensor $\mathcal{T},$ denoted by
$\lambda(\mathcal{T})$.

\bigskip The concept of power hypergraphs was introduced in
\cite{HuQiShao2013}.

\begin{definition}
\label{power graph}Let $G=(V(G),E(G))$ be an ordinary graph. For every
$k\geq2$, the $k$th power of $G$, $G^{k}:=(V(G^{k}),E(G^{k}))$ is defined as
the $k$-uniform hypergraph with the edge set
\[
E(G^{k}):=\{e\cup\{i_{e,1},\cdot\cdot\cdot,i_{e,k-2}\}\text{ }|\text{ }e\in
E(G)\}
\]
and the vertex set
\[
V(G^{k}):=V(G)\cup(\cup_{e\in E(G)}\{i_{e,1},\cdot\cdot\cdot,i_{e,k-2}\}).
\]

\end{definition}

\bigskip For convenience here $G^{2}=G$. In \cite{HuQiShao2013}, the $k$-th
power of path, and cycle is called loose path, and loose cycle, respectively.
Denote by $S_{m}$ the star with $m$ edges. The $k$-th power of star is called
sunflower in \cite{HuQiShao2013}, or hyperstar in \cite{LiShaoQi}.

\begin{definition}
\label{sunflower} \cite{HuQiShao2013} Let $G=(V,E)$ be a $k$-uniform
hypergraph. If there is a disjoint partition of the vertex set $V$ as
$V=V_{0}\cup V_{1}\cup\cdot\cdot\cdot\cup V_{d}$ such that $|V_{0}|=1$ and
$|V_{1}|=\cdot\cdot\cdot=|V_{d}|=k-1,$ and $E=\{V_{0}\cup V_{i}$ $|$
$i\in\lbrack d]\}$, then $G$ is called a sunflower. The degree $d$ of the
vertex in $V_{0}$, which is called the heart, is the size of the sunflower.
Denote by $S_{d}^{k}$ the $k$-uniform sunflower of size $d.$
\end{definition}

For even $k$, when $G$ is a cycle or star, Hu, Qi and Shao proved that
$\{\lambda(\mathcal{Q}(G^{k}))\}$ is a strictly decreasing sequence in
\cite{HuQiShao2013}. They believed that it is true for any graph $G,$ see
Conjecture 4.1 of \cite{HuQiShao2013}. This phenomena was also observed in
\cite{YueZhangLu} when $G$ is a path (namely, $G^{k}$ is a loose path).

\begin{conjecture}
\cite{HuQiShao2013} \label{con}Let $G$ be an ordinary graph, $k=2r$ be even
and $G^{k}$ be the $k$-th power hypergraph of $G$. Then $\{\lambda
(\mathcal{L}(G^{k}))=\lambda(\mathcal{Q}(G^{k}))\}$ is a strictly decreasing sequence.
\end{conjecture}

For $t\geq1$ let $tS_{1}^{k}$ be $t$ disjoint union of $S_{1}^{k}.$ We may
point out that when $G=tS_{1}^{2}$, we have $\lambda(\mathcal{L}%
(G^{k}))=\lambda(\mathcal{Q}(G^{k}))=2$ for any $k\geq2.$ Namely, in this case
Conjecture \ref{con} is false.

Let $H$ be a uniform hypergraph and $H\neq tS_{1}^{k}$, and $H^{\prime}$ be
obtained from $H$ by inserting a new pendent vertex in each edge. In Section
3, we will prove that $\lambda(\mathcal{Q(}H^{^{\prime}}\mathcal{)}%
)<\lambda(\mathcal{Q(}H\mathcal{)})$ in Theorem \ref{main 1}. So for any
ordinary graph $G\neq tS_{1}\ $(maximum degree $\Delta\geq2),\{\lambda
(\mathcal{Q}(G^{k}))\}$ is a strictly decreasing sequence$,$ which affirm
Conjecture \ref{con} for $\lambda(\mathcal{Q}(G^{k})).$ We also determine the
value lim$_{k\rightarrow\infty}\lambda(\mathcal{Q}(G^{k}))$ in Theorem
\ref{main 2}.

For an ordinary graph $G,$ the definition for $k$-th power hypergraph $G^{k}$
has been generalized by Khan and Fan in \cite{KhanFan2015}.

\begin{definition}
\label{generalized power hypergraph}Let $G=(V,E)$ be an ordinary graph. For
any $k\geq3$ and $1\leq s\leq k/2$. For each $v\in V$ (and $e\in E$), let
$V_{v}$ (and $V_{e}$) be a new vertex set with $s$ (and $k-2s$) elements such
that all these new sets are pairwise disjoint. Then the generalized power of
$G,$ denoted by $G^{k,s}$, is defined as the $k$-uniform hypergraph with the
vertex set
\[
V(G^{k,s})=\left(  \bigcup_{v\in V}V_{v}\right)  \bigcup\left(  \bigcup_{e\in
E}V_{e}\right)
\]
and edge set
\[
E(G^{k,s})=\{V_{u}\cup V_{v}\cup V_{e}:e=\{u,v\}\in E \}.
\]

\end{definition}

If $s=1$, then $G^{k,s}$ is exactly the $k$th power hypergraph $G^{k}$. The
eigenvalues properties about $\mathcal{A(}G^{k}\mathcal{)}$ was discussed in
\cite{ZhouSunWangBu}, and $\mathcal{A(}G^{k,k/2}\mathcal{)}$ was discussed in
\cite{KhanFan2015} and \cite{KhanFan[ArX]}. In Section 4, we will prove some
eigenvalues properties about $\mathcal{A(}G^{k,s}\mathcal{)}$, which
generalize some known results.

\section{Auxiliary results for nonnegative tensors and H-spectrum of
hypergraphs}

In \cite{FGH}, the weak irreducibility of nonnegative tensors was defined. It
was proved in \cite{FGH} and \cite{Yang2014} that a $k$-uniform hypergraph $G$
is connected if and only if its adjacency tensor $\mathcal{A(}G\mathcal{)}$
(and so $\mathcal{Q(}G\mathcal{)}$) is weakly irreducible.

Let $\mathcal{T}$ be a $k$th-order $n$-dimensional nonnegative tensor. The
spectral radius of $\mathcal{T}$ is defined as (see \cite{LiShaoQi},
\cite{KhanFan2015} and \cite{KhanFan[ArX]})
\[
\rho(\mathcal{T})=max\{|\mu|:\mu\text{ is an eigenvalue of }\mathcal{T}\}.
\]
Part of Perron-Frobenius theorem for nonnegative tensors is stated in the
following for reference.

\begin{theorem}
\label{Perron-Frobenius}\cite{ChangPZ} \cite{Yang2011} Let $\mathcal{T}$ be a
nonnegative tensor. Then we have the following statements.\newline(1).
$\rho(\mathcal{T})$ is an eigenvalue of $\mathcal{T}$ with a nonnegative
eigenvector $x$ corresponding to it;\newline(2). If $\mathcal{T}$ is weakly
irreducible, then $x$ is positive, and for any eigenvalue $\mu$ with
nonnegative eigenvector, $\mu$ $=\rho(\mathcal{T})$ holding;\newline(3). The
nonnegative eigenvector $x$ corresponding to $\rho(\mathcal{T})$ is unique up
to a constant multiple.
\end{theorem}

In virtue of (1) of Theorem \ref{Perron-Frobenius}, the largest H-eigenvalue
of $\mathcal{A}(G)$ (or $\mathcal{Q}(G)$) is exactly the spectral radius of
$\mathcal{A}(G)$ (or $\mathcal{Q}(G)$) . For weakly irreducible nonnegative
$\mathcal{T}$ of order $k,$ the positive eigenvector $x$ with $||x||_{k}=1$
corresponding to $\rho(\mathcal{T})$ (i.e., the largest H-eigenvalue) is
called the principal eigenvector of $\mathcal{T}$ in this paper.

\begin{lemma}
\label{inquality for nonpositive tensor} \cite{KhanFan2015}Suppose that
$\mathcal{T}$ is a weakly irreducible nonnegative tensor of order $k$. If
there exists a nonnegative vector $y$ such that $\mathcal{T}y\leq\mu
y^{[k-1]}$ and $(\mathcal{T}y)_{i}<\mu y_{i}^{k-1}$ holding for some $i$, then
$\lambda(\mathcal{T})<\mu$.
\end{lemma}

The H-spectrum of a real tensor $\mathcal{T}$, denoted by $Hspec(\mathcal{T}%
),$ is defined to be the set of distinct H-eigenvalues of $\mathcal{T}$
\cite{ShaoShanWu}$.$ Namely,%

\[
Hspec(\mathcal{T})=\{\mu\text{
$\vert$
}\mu\text{ is an H-eigenvalue of }\mathcal{T}\}.
\]

\begin{lemma}
\label{non-connected}\bigskip Let $G=\cup_{i=1}^{t}G_{i},$ where $G_{i}$ is a
connected uniform hypergraph. Then
\begin{equation}
Hspec(\mathcal{L}(G))=\bigcup_{i=1}^{t}Hspec(\mathcal{L}(G_{i})), \label{L1}%
\end{equation}
and so%
\[
\lambda(\mathcal{L}(G))=\max_{1\leq i\leq t}\{\lambda(\mathcal{L}(G_{i}))\}.
\]

\end{lemma}

\bigskip

\begin{proof}
Without loss of the generality, we may assume that the vertices of $G$ are
ordered in such a way that if $i<j$, then any vertex in $G_{i}$ precedes any
vertex in $G_{j}$. \bigskip Let $x$ be a column vector of dimension $|V(G)|$.
We write $x$ in the following block form
\begin{equation}
x=(x_{1}^{T},x_{2}^{T},\cdots,x_{t}^{T})^{T}, \label{L3}%
\end{equation}
where $x_{i}$ is a column vector corresponding to the vertices of $G_{i}$.
Then it is not difficult to see that
\begin{equation}
\mathcal{L}(G)x=((\mathcal{L}(G_{1})x_{1})^{T},(\mathcal{L}(G_{2})x_{2}%
)^{T},\cdots,(\mathcal{L}(G_{t})x_{t})^{T})^{T}. \label{L4}%
\end{equation}
Now we prove Eq.(\ref{L1}). If $\lambda\in Hspec(\mathcal{L}(G))$ with a real
eigenvector $x$ as in (\ref{L3}), where $x_{j}\neq0$. Then by $\mathcal{L}%
(G)x=\lambda x^{[k-1]}$ and Eq.(\ref{L4}) we have
\[
\mathcal{L}(G_{j})x_{j}=\lambda x_{j}^{[k-1]}.
\]
Thus
\begin{equation}
\lambda\in Hspec(\mathcal{L}(G_{j}))\subseteq\bigcup_{i=1}^{t}%
Hspec(\mathcal{L}(G_{i})). \label{L5}%
\end{equation}
On the other hand, if $\lambda\in\bigcup_{i=1}^{t}Hspec(\mathcal{L}(G_{i}))$,
say, $\lambda\in Hspec(\mathcal{L}(G_{j}))$ for some $1\leq j\leq t$ with a
real eigenvector $x_{j}$. Take $x_{i}=0$ for all $i\neq j$ and take $x$ as in
(\ref{L3}). Then by Eq.(\ref{L4}) we can verify that \bigskip$\mathcal{L}%
(G)x=\lambda x^{[k-1]}$, thus $\lambda\in Hspec(\mathcal{L}(G))$. Combining
these two aspects, we obtain (\ref{L1}).

Particularly, we have%
\[
\lambda(\mathcal{L}(G))=\max_{1\leq i\leq t}\{\lambda(\mathcal{L}(G_{i}))\}.
\]

Similarly, we may prove that these results are also true for $\mathcal{Q}(G)$
and $\mathcal{A}(G).$
\end{proof}

\section{Largest H-eigenvalue of signless Laplacian tensor of $G^{k}$}

In this section we will prove that for any graph $G$ with maximum degree
$\Delta\geq2$, $\{\lambda(\mathcal{Q}(G^{k}))\}$ is a strictly decreasing
sequence$.$ First we will prove a more general result by constructing a new
vector and using Lemma \ref{inquality for nonpositive tensor}.

\begin{theorem}
\label{main 1} Let $H$ be a $k$-uniform ($k\geq2$) hypergraph, and $H^{\prime
}$ be obtained from $H$ by inserting a new pendent vertex in each edge. Then
$\lambda(\mathcal{Q(}H^{\prime}\mathcal{)})\leq\lambda(\mathcal{Q(}%
H\mathcal{)}),$ equality holding if and only if $H=tS_{1}^{k}$ for some $t$.
\end{theorem}

\begin{proof}
If $H=tS_{1}^{k}$ for some $t,$ then $\lambda(\mathcal{Q(}H^{\prime
}\mathcal{)})=\lambda(\mathcal{Q(}H\mathcal{)})=2.$ We suppose that $H\neq
tS_{1}^{k}$ for any $t.$

Denote by \
\[
E(H)=\{e_{1},e_{2},\cdot\cdot\cdot,e_{m}\},
\]
and let $v_{i}$ be the new pendent vertex inserted in $e_{i}$ for any $1\leq
i\leq m,$ i.e.,
\[
V(H^{\prime})=V(H)\cup\{v_{1},v_{2},\cdot\cdot\cdot,v_{m}\}.
\]

(1). First suppose that $H$ is a connected, so $\mathcal{Q}(H)$ is weakly
irreducible. Let $x$ be the principal eigenvector to $\lambda(\mathcal{Q}%
(H)),$ namely,
\[
\mathcal{Q}(H)x=\lambda(\mathcal{Q}(H))x^{[k-1]}.
\]
Now we construct a new vector $y$ (of dimension $|V(H^{\prime})|$) from $x$ by
adding $m$ components. If $w\in V(H),$ set $y_{w}=x_{w};$ if $w=v_{i},$ i.e.,
$w$ is a new pendent vertex inserted in $e_{i},$ set $y_{w}=\min\{x_{u}$ $|$
$u\in e_{i}\}$.

Now we will show $\mathcal{Q}(H^{\prime})y\leq\lambda(\mathcal{Q}(H))y^{[k]}.$

For any vertex $w\in V(H)$ we have
\begin{align}
(\mathcal{Q}(H^{\prime})y)_{w}  &  =d_{H^{^{\prime}}}(w)y_{w}^{k}%
+\sum_{\{w,v_{i},t_{2},\cdots,t_{k}\}\in E(H^{^{\prime}})}y_{v_{i}}y_{t_{2}%
}\cdot\cdot\cdot y_{t_{k}}\nonumber\\
&  \leq d_{H}(w)x_{w}^{k}+\sum_{\{w,v_{i},t_{2},\cdots,t_{k}\}\in
E(H^{^{\prime}})}x_{w}x_{t_{2}}\cdot\cdot\cdot x_{t_{k}}\\
&  =x_{w}\Big[d_{H}(w)x_{w}^{k-1}+\sum_{\{w,t_{2},\cdots,t_{k}\}\in
E(H)}x_{t_{2}}\cdot\cdot\cdot x_{t_{k}}\Big]\nonumber\\
&  =x_{w}(\mathcal{Q}(H)x)_{w}\nonumber\\
&  =x_{w}\lambda(\mathcal{Q}(H))x_{w}^{k-1}\nonumber\\
&  =\lambda(\mathcal{Q}(H))x_{w}^{k}\nonumber\\
&  =\lambda(\mathcal{Q}(H))y_{w}^{k}.\nonumber
\end{align}
Ineq. (4) is due to the fact $y_{v_{i}}\leq x_{w}.$ Furthermore, if
$x_{w}>y_{v_{i}},$ namely, $x_{w}>\min\{x_{u}$ $|$ $u\in e_{i}\}$ for some
edge $e_{i}$ containing $w,$ Inequality (4) becomes strict.

For $w=v_{i}$ for some $1\leq i\leq m,$ i.e., $w$ is a new pendent vertex
inserted in $e_{i},$ we suppose $e_{i}=\{w_{1},w_{2},\cdot\cdot\cdot,w_{k}\}$
and $x_{w_{1}}=\min\{x_{w_{1}},x_{w_{2}},...,x_{w_{k}}\}.$ Then we have
\begin{align}
(\mathcal{Q}(H^{\prime})y)_{w}  &  =y_{w}^{k}+y_{w_{1}}y_{w_{2}}\cdot
\cdot\cdot y_{w_{k}}\nonumber\\
&  =x_{w_{1}}^{k}+x_{w_{1}}x_{w_{2}}\cdot\cdot\cdot x_{w_{k}}\nonumber\\
&  \leq x_{w_{1}}\Big[d_{H}(w_{1})x_{w_{1}}^{k-1}+x_{w_{2}}\cdot\cdot\cdot
x_{w_{k}}\Big]\\
&  \leq x_{w_{1}}(\mathcal{Q}(H)x)_{w_{1}}\nonumber\\
&  =\lambda(\mathcal{Q}(H))x_{w_{1}}^{k}\nonumber\\
&  =\lambda(\mathcal{Q}(H))y_{w}^{k}.\nonumber
\end{align}
Ineq. (5) is due to the fact $d_{H}(w_{1})\geq1.$ Furthermore, if $d_{H}%
(w_{1})>1,$ Ineq. (5) becomes strict. So we have proved $\mathcal{Q}%
(H^{\prime})y\leq\lambda(\mathcal{Q}(H))y^{[k]}.$

If there exists $\{w,w^{\prime}\}\subseteq e_{i}$ for some $e_{i}\in E(H)$
such that $x_{w}>x_{w^{\prime}},$ then $x_{w}>\min\{x_{u}$ $|$ $u\in e_{i}\}$,
and then Ineq. (4) becomes strict. Now we suppose all the vertices in each
edge have the equal corresponding component in $x.$ Furthermore, since $H$ is
connected, \ we know that all the components of $x$ are equal, thus $H$ is
regular. We suppose $H$ is a $d$-regular. Since $H\neq S_{1}^{k}$, we have
$d\geq2.$ Now for some edge $e_{i}$, set $x_{w_{1}}=\min\{x_{u}$ $|$ $u\in
e_{i}\},$ then $d_{H}(w_{1})=d>1,$ and Ineq. (5) becomes a strict one.

Thus we have proved that $(\mathcal{Q}(H^{\prime})y)_{i}<\lambda
(\mathcal{Q}(H))y_{i}^{[k]}$ for some $i,$ therefore, $\lambda(\mathcal{Q}%
(H^{\prime}))<\lambda(\mathcal{Q}(H))$ by Lemma
\ref{inquality for nonpositive tensor}.

(2) If $H=H_{1}\cup H_{2}\cup\cdot\cdot\cdot\cup H_{t},$ where $H_{i}$ is a
connected component of $H$ for $1\leq i\leq t$ and $t\geq2,$ then
\[
H^{\prime}=H_{1}^{\prime}\cup H_{2}^{\prime}\cup\cdot\cdot\cdot\cup
H_{t}^{\prime}.
\]
Since $H\neq tS_{1}^{k},$ we may suppose that $H_{i}\neq S_{1}^{k}$ for $1\leq
i\leq t^{\prime}\leq t,$ and $t^{\prime}\geq1.$ Then for $1\leq i\leq
t^{\prime}$ we have $\lambda(\mathcal{Q}(H_{i}^{\prime}))<\lambda
(\mathcal{Q}(H_{i}))$ by the above arguments. It is obvious that
$\lambda(\mathcal{Q}(H_{i}))>\lambda(\mathcal{Q}(S_{1}^{k}))$ and
$\lambda(\mathcal{Q}(H_{i}^{\prime}))>\lambda(\mathcal{Q}(S_{1}^{k+1}))$ hold.
By Lemma \ref{non-connected} we have
\[
\lambda(\mathcal{Q}(H^{\prime}))=\max_{1\leq i\leq t^{\prime}}\{\lambda
(\mathcal{Q}(H_{i}^{\prime}))\}<\max_{1\leq i\leq t^{\prime}}\{\lambda
(\mathcal{Q}(H_{i}))\}=\lambda(\mathcal{Q}(H)).
\]

The proof is completed.
\end{proof}

\bigskip By Theorem \ref{main 1}, we have the following result for
$\{\lambda(\mathcal{Q}(G^{k}))\},$ which affirm Conjecture \ref{con} for
$\mathcal{Q}(G^{k}).$

\begin{theorem}
\label{main 1 Q} Let $G$ be an ordinary graph with maximum $\Delta\geq2$. When
$k\geq2$ we have $\lambda(\mathcal{Q}(G^{k+1}))<\lambda(\mathcal{Q}(G^{k}))$.
\end{theorem}

The notion of odd-bipartite even-uniform hypergraphs was introduced in
\cite{HuQi2014DAM}.

\begin{definition}
\cite{HuQi2014DAM} Let $k$ be even and $G=(V,E)$ be a $k$-uniform hypergraph.
It is called odd-bipartite if either it is trivial (i.e., $E=\emptyset$) or
there is a disjoint partition of the vertex set $V$ as $V=V_{1}\cup V_{2}$
such that $V_{1},V_{2}\neq\emptyset$ and every edge in $E$ intersects $V_{1}$
with exactly an odd number of vertices.
\end{definition}

For even uniform odd-bipartite hypergraph, the following result was proved in
\cite{HuQiXie2015} (see \ Theorem 5.8 of \cite{HuQiXie2015}), or in
\cite{ShaoShanWu} (see \ Theorem 2.2 of \cite{ShaoShanWu}).

\begin{lemma}
\label{L and Q} \cite{HuQiXie2015}\bigskip\ \cite{ShaoShanWu} Let $G$ be a
connected even uniform odd-bipartite hypergraph. Then $\lambda(\mathcal{L}%
(G))=\lambda(\mathcal{Q}(G))$.
\end{lemma}

In fact Lemma \ref{L and Q} is also true for general even uniform
odd-bipartite hypergraph $G$, see Lemma \ref{L and Q II}.

\begin{lemma}
\label{L and Q II} Let $G$ be an even uniform odd-bipartite hypergraph. Then
$\lambda(\mathcal{L}(G))=\lambda(\mathcal{Q}(G))$.
\end{lemma}

\begin{proof}
By Lemma \ref{L and Q}, we only need to consider the case that $G$ is not
connected. Set $G=\cup_{i=1}^{t}G_{i},$ where $G_{i}$ is a connected component
of $G$ for $1\leq i\leq t$ and $t\geq2.$ Since $G$ is even uniform and
odd-bipartite, each $G_{i}$ is connected even uniform and odd-bipartite. Thus
by Lemma \ref{L and Q} we have $\lambda(L(G_{i}))=\lambda(Q(G_{i}%
))\ (i=1,\cdots,t),$ and moreover, by Lemma \ref{non-connected} we have
\[
\lambda(\mathcal{L}(G))=\max_{1\leq i\leq t}\{\lambda(\mathcal{L}%
(G_{i}))\}=\max_{1\leq i\leq t}\{\lambda(\mathcal{Q}(G_{i}))\}=\lambda
(\mathcal{Q}(G)).
\]

The proof is completed.
\end{proof}

\begin{remark}
We can mention here that the condition $\lambda(\mathcal{L}(G))=\lambda
(\mathcal{Q}(G))$ does not imply that $G$ is an even uniform odd-bipartite
hypergraph. In fact take $G=G_{1}\cup G_{2}$, where $G_{1}$ is not
odd-bipartite and $G_{2}$ is a sunflower with size $\Delta$ satisfying
\[
\Delta>\lambda(\mathcal{Q}(G_{1}))\geq\lambda(\mathcal{L}(G_{1})).
\]
From Proposition 3.2 of \cite{HuQiXie2015} and Lemma \ref{L and Q}, we know
that
\[
\lambda(\mathcal{L}(G_{2}))=\lambda(\mathcal{Q}(G_{2}))>\Delta.
\]
Then $G$ is not odd-bipartite (since $G_{1}$ is not), but we have
\[
\lambda(\mathcal{L}(G))=\max_{1\leq i\leq2}\{\lambda(\mathcal{L}%
(G_{i}))\}=\lambda(\mathcal{L}(G_{2}))=\lambda(\mathcal{Q}(G_{2}))=\max_{1\leq
i\leq2}\{\lambda(\mathcal{Q}(G_{i}))\}=\lambda(\mathcal{Q}(G)).
\]

\end{remark}

Obviously, when $k$ is even and $k\geq4$, the $k$-th power hypergraph $G^{k}$
is odd-bipartite. Then Lemma \ref{L and Q II} and Theorem \ref{main 1 Q} imply
that Conjecture \ref{con} is true for $\mathcal{L}(G^{k}).$

\begin{theorem}
\label{for Laplacian} Let $G=(V,E)$ be an ordinary graph with maximum
$\Delta\geq2$, $k=2r$ be even and $G^{k}$ be the $k$-power hypergraph of $G$.
Then $\{\lambda(\mathcal{L}(G^{k}))=\lambda(\mathcal{Q}(G^{k}))\}$ is a
strictly decreasing sequence.
\end{theorem}

The value $\lim_{k\rightarrow\infty}\lambda(\mathcal{Q}(G^{k}))$ was
determined for a regular graph $G$ by Zhou et al. in \cite{ZhouSunWangBu}.

\begin{lemma}
\label{for regular lim}\cite{ZhouSunWangBu} For any $d$-regular graph $G$ with
$d\geq2,$ we have $\lim_{k\rightarrow\infty}\lambda(\mathcal{Q}(G^{k}))=d.$
\end{lemma}

Now we will prove that when $k$ goes to infinity, $\lambda(\mathcal{Q}%
(S_{d}^{k}))$ converges to $d,$ the maximum degree of $S_{d}^{k}.$

\begin{lemma}
\label{lemma for sunflower}When $k\geq2,d\geq2$ we have $\lim_{k\rightarrow
\infty}\lambda(\mathcal{Q}(S_{d}^{k}))=d.$
\end{lemma}

\begin{proof}
Write $\lambda_{k}=\lambda(\mathcal{Q}(S_{d}^{k}))$ for short. Let $x$ be the
principal eigenvector of $\mathcal{Q}(S_{d}^{k})$ corresponding to
$\lambda_{k}$. Let $a$ be the component of $x$ corresponding to the heart of
$S_{d}^{k}.$ By the symmetry of the pendent vertices in the same edge, we see
that they have the same component in $x.$ Furthermore, by the uniqueness of
$x$ (see (3) of Theorem \ref{Perron-Frobenius}), we see that all the pendent
vertices in $S_{d}^{k}$ have the same component in $x,$ say $b.$ Then
$\lambda_{k}$ satisfies the following equations%
\[
\left\{
\begin{array}
[c]{ll}%
\lambda_{k}a^{k-1} & =da^{k-1}+db^{k-1},\\
\lambda_{k}b^{k-1} & =b^{k-1}+ab^{k-2}.
\end{array}
\right.
\]
By eliminations of $a$ and $b,$ we obtain%

\[
(\lambda_{k}-d)(\lambda_{k}-1)^{k-1}-d=0.
\]
Set
\[
f_{k}(\lambda)=(\lambda-d)(\lambda-1)^{k-1}-d,
\]
then $\lambda_{k}$ is the largest real root of the equation $f_{k}(\lambda)=0.$

Particularly,
\[
(\lambda_{k+1}-d)(\lambda_{k+1}-1)^{k}=d.
\]
Since $f_{k}(d)=-d<0,$ and lim$_{\lambda\rightarrow+\infty}f_{k}%
(\lambda)=+\infty,$ we have $\lambda_{k}>d.$

So $\{\lambda_{k}\}$ is a strictly decreasing sequence and $\lambda_{k}>d$
when $d\geq2.$ Thus $\lim_{k\rightarrow\infty}\lambda_{k}$ exists. From
\[
(\lambda_{k}-d)(\lambda_{k}-1)^{k-1}=d,
\]
we have
\[
\lim_{k\rightarrow\infty}\lambda_{k}-d=\frac{d}{\lim_{k\rightarrow\infty
}(\lambda_{k}-1)^{k-1}}=0,
\]
thus $\lim_{k\rightarrow\infty}\lambda_{k}=d,$ i.e., $\lim_{k\rightarrow
\infty}\lambda(\mathcal{Q}(S_{d}^{k}))=d$ holds.
\end{proof}

To determine the value $\lim_{k\rightarrow\infty}\lambda(\mathcal{Q}(G^{k}))$
for a general graph $G$, we first cite a result just for a graph due to
K$\ddot{o}$ing.

\begin{lemma}
\label{for regular induced}\cite{Koing1963} Every graph $G$ of maximum degree
$\Delta$ is an induced subgraph of some $\Delta$-regular graph.
\end{lemma}

\begin{theorem}
\label{main 2} Let $G$ be an ordinary graph with maximum degree $\Delta\geq2.$
Then we have $\lim_{k\rightarrow\infty}\lambda(\mathcal{Q}(G^{k}))=\Delta.$
\end{theorem}

\begin{proof}
We have proved that $\{\lambda(\mathcal{Q}(G^{k}))\}$ is a strictly decreasing
sequence by Theorem \ref{main 1 Q}. Obviously, $\lambda(\mathcal{Q}(G^{k}))>0$
for any $k\geq2.$ Thus $\lim_{k\rightarrow\infty}\lambda(\mathcal{Q}(G^{k}))$ exists.

It is known that if $F^{\prime}$ is a sub-hypergraph of $F,$ then
$\lambda(\mathcal{Q}(F^{\prime}))\leq\lambda(\mathcal{Q}(F))$ (see Proposition
4.5 in \cite{HuQiXie2015}).

Since $G$ has maximum degree $\Delta,$ $G$ contains star $S_{\Delta}$ as a
sub-graph, and then $G^{k}$ contains the sunflower $S_{\Delta}^{k}$ as a
sub-hypergraph. Thus $\lambda(\mathcal{Q}(G^{k}))\geq\lambda(\mathcal{Q}%
(S_{\Delta}^{k})).$ Furthermore Lemma \ref{lemma for sunflower} implies that
\[
\lim_{k\rightarrow\infty}\lambda(\mathcal{Q}(G^{k}))\geq\lim_{k\rightarrow
\infty}\lambda(\mathcal{Q}(S_{\Delta}^{k}))=\Delta.
\]
\ On the other hand, by Lemma \ref{for regular induced}, $G$ is a subgraph of
some $\Delta$-regular graph $F.$ Then $G^{k}$ is a sub-hypergraph of $F^{k}$.
Thus $\lambda(\mathcal{Q}(G^{k}))\leq\lambda(\mathcal{Q}(F^{k})).$ Furthermore
by Lemma \ref{for regular lim} we have
\[
\lim_{k\rightarrow\infty}\lambda(\mathcal{Q}(G^{k}))\leq\lim_{k\rightarrow
\infty}\lambda(\mathcal{Q}(F^{k}))=\Delta.
\]
So we obtain
\[
\lim_{k\rightarrow\infty}\lambda(\mathcal{Q}(G^{k}))=\Delta.
\]

\end{proof}

\section{Largest H-eigenvalue of adjacency tensor of $G^{k,s}$}

In \cite{ZhouSunWangBu}, it was proved that $\lambda(\mathcal{A}%
(G^{k}))=\lambda(A(G))^{\frac{2}{k}}$; in \cite{KhanFan2015} it was proved
that $\lambda(\mathcal{A}(G^{k,k/2}))=\lambda(A(G))$. By using the technique
provided in \cite{ZhouSunWangBu}, we will prove a general case.

\begin{theorem}
\label{main 3}If $\mu\neq0$ is an eigenvalue of the adjacency matrix $A(G)$ of
graph $G,$ then $\mu^{\frac{2s}{k}}$ is an eigenvalue of the adjacency tensor
$\mathcal{A}(G^{k,s}).$ Moreover $\lambda(\mathcal{A}(G^{k,s}))=\lambda
(A(G))^{\frac{2s}{k}.}$
\end{theorem}

\begin{proof}
Suppose that $x$ is an eigenvector of the eigenvalue $\mu\neq0$ of $A(G)$. As
shown in Definition \ref{generalized power hypergraph}, for any edge
$e=\{u,v\}$ denote by $V_{u}\cup V_{v}\cup V_{e}$ the corresponding edge of
$G^{k,s}$.

Now we construct a new vector $y$ (of dimension $|V(G^{k,s})|$) from $x$ by
adding components. Set%
\[
y_{w}=\left\{
\begin{array}
[c]{ll}%
(x_{v})^{\frac{2}{k}} & \text{if }w\in V_{v}\text{ for some }v,\\
(\mu^{-1}x_{u}x_{v})^{\frac{1}{k}} & \text{if }w\in V_{e}\text{ for some edge
}e=\{u,v\}.
\end{array}
\right.
\]
Now we will show $\mathcal{A}(G^{k,s})y=$ $\mu^{\frac{2s}{k}}y^{[k-1]}$ holding.

For any $w\in V_{v}$ for some $v,$ by the formula
\[
\sum_{\{u,v\}\in E(G)}x_{u}=\mu x_{v},
\]
we have
\begin{align}
(\mathcal{A}(G^{k,s})y)_{w}  &  =\sum_{\{u,v\}\in E(G)}(x_{v})^{\frac
{2(s-1)}{k}}(x_{u})^{\frac{2s}{k}}(\mu^{-1}x_{u}x_{v})^{\frac{k-2s}{k}%
}\nonumber\\
&  =\mu^{\frac{2s}{k}-1}(x_{v})^{\frac{k-2}{k}}\sum_{\{u,v\}\in E(G)}x_{u}\\
&  =\mu^{\frac{2s}{k}}(x_{v})^{\frac{2(k-1)}{k}}\nonumber\\
&  =\mu^{\frac{2s}{k}}y_{w}^{k-1}.\nonumber
\end{align}

For $w\in V_{e}$ for any edge $e=\{u,v\},$ we have
\begin{align}
(\mathcal{A}(G^{k,s})y)_{w}  &  =(x_{u})^{\frac{2s}{k}}(x_{v})^{\frac{2s}{k}%
}(\mu^{-1}x_{u}x_{v})^{\frac{k-2s-1}{k}}\nonumber\\
&  =\mu^{\frac{2s}{k}}(\mu^{-1}x_{u}x_{v})^{\frac{k-1}{k}}\nonumber\\
&  =\mu^{\frac{2s}{k}}y_{w}^{k-1}.\nonumber
\end{align}

Hence $\mu^{\frac{2s}{k}}$ is an eigenvalue of $\mathcal{A}(G^{k,s})$ with
eigenvector $y.$

If $G$ is connected and $\mu=\lambda(A(G)),$ then we may choose $x$ as a
positive eigenvector of $\lambda(A(G))$ by Perron-Frobenius theorem for
irreducible nonnegative matrix. In this case $y$ is a positive eigenvector of
the eigenvalue $\lambda(A(G))^{\frac{2s}{k}}$ of tensor $\mathcal{A}%
(G^{k,s}).$ In virtue of (2) of Theorem \ref{Perron-Frobenius} (or see Lemma
15 of \cite{ZhouSunWangBu}), we have
\[
\lambda(\mathcal{A}(G^{k,s}))=\lambda(A(G))^{\frac{2s}{k}.}%
\]

If $G=G_{1}\cup G_{2}\cup\cdot\cdot\cdot\cup G_{t},$ where $G_{i}$ is a
connected component of $G$ for $1\leq i\leq t$ and $t\geq2,$ then by Lemma
\ref{non-connected}
\[
\lambda(\mathcal{A}(G^{k,s}))=\max_{1\leq i\leq t}\{\lambda(\mathcal{A}%
(G_{i}^{k,s})\}=\max_{1\leq i\leq t}\{\lambda(\mathcal{A}(G_{i})^{\frac{2s}%
{k}}\}=\lambda(A(G))^{\frac{2s}{k}}.
\]

The proof is completed.
\end{proof}

Take $s=1,$ or $s=k/2$ for even $k,$ and noting that $G^{k}=G^{k,1},$ we have
Corollary \ref{Bu} and Corollary \ref{Fan}.

\begin{corollary}
\label{Bu}\cite{ZhouSunWangBu}If $\mu\neq0$ is an eigenvalue of the adjacency
matrix $A(G)$ of graph $G,$ then $\mu^{\frac{2}{k}}$ is an eigenvalue of the
adjacency tensor $\mathcal{A}(G^{k})$ of the hypergraph $G^{k}.$ Moreover
$\lambda(\mathcal{A}(G^{k}))=\lambda(A(G))^{\frac{2}{k}.}$
\end{corollary}

\begin{corollary}
\label{Fan}\cite{KhanFan2015}Let $G$ be a connected ordinary graph, and let
$x>0$ be vector defined on $V(G)$. Let $x>0$ be a vector defined on
$V(G^{k,k/2})$ such that $x_{u}=x_{v}^{\frac{2}{k}}$ for each vertex $u\in
V_{v}$. Then $x$ is an eigenvector of $A(G)$ corresponding to $\lambda(A(G))$
if and only if $x$ is an eigenvector of $\mathcal{A}(G^{k,k/2})$ corresponding
to $\lambda(\mathcal{A}(G^{k,k/2}))$. Hence $\lambda(\mathcal{A}%
(G^{k,k/2}))=\lambda(A(G)).$
\end{corollary}

\begin{corollary}
Let $G$ be an ordinary graph with maximum $\Delta\geq2,s\geq1$ be a fixed
integer. Then $\lambda(\mathcal{A}(G^{k+1,s}))<\lambda(\mathcal{A}(G^{k,s})),$
and $\lim_{k\rightarrow\infty}\lambda(\mathcal{A}(G^{k,s}))=1.$
\end{corollary}

\end{document}